\documentclass[11pt]{amsart}
\usepackage{amsthm, amssymb,latexsym}
\usepackage{enumerate}
\usepackage{graphicx}

\def\Log{\text{\rm Log\,}}

\def\dis{\displaystyle}

\makeatletter
\@namedef{subjclassname@2010}{%F
  \textup{2010} Mathematics Subject Classification}
\makeatother

\newtheorem{thm}{Theorem}

\newtheorem{prop}[thm]{Proposition}

\newtheorem{conj}{Conjecture}

\begin{document}

\baselineskip=15pt

\title[Estimates for the Bergman Kernel]
{Estimates for the Bergman Kernel \\
and the Multidimensional Suita Conjecture}

\thanks {The first named author was supported by the Ideas Plus 
grant 0001/ID3/2014/63 of the Polish Ministry Of Science and 
Higher Education and the second named author by 
the Polish National Science Centre grant  
2011/03/B/ST1/04758}

\def\nl{\newline\phantom{a}\hskip 7pt}

\author{Zbigniew B\l ocki, W\l odzimierz Zwonek}
\address{Uniwersytet Jagiello\'nski \nl Instytut Matematyki
\nl \L ojasiewicza 6 \nl 30-348 Krak\'ow \nl Poland
\nl {\rm Zbigniew.Blocki@im.uj.edu.pl 
\nl Wlodzimierz.Zwonek@im.uj.edu.pl}}

\begin{abstract}
We study the lower bound for the Bergman kernel in terms
of volume of sublevel sets of the pluricomplex Green 
function. We show that it implies a bound
in terms of volume of the Azukawa indicatrix which 
can be treated as a multidimensional version of the 
Suita conjecture. We also prove that the corresponding upper 
bound holds for convex domains and discuss it in bigger
detail on some convex complex ellipsoids.
\end{abstract}

\makeatletter
\@namedef{subjclassname@2010}{%
  \textup{2010} Mathematics Subject Classification}
\makeatother

%\subjclass[2010]{}

%\keywords{}

\maketitle

\section{Introduction and Statement of Main Results}

Let $\Omega$ be a pseudoconvex domain in $\mathbb C^n$. The following
lower bound for the Bergman kernel in terms of the pluricomplex
Green function was recently proved in \cite{B5} using 
methods of the $\bar\partial$-equation: for any $t\leq 0$ 
and $w\in\Omega$ one has
\begin{equation}\label{est1}
K_\Omega(w)\geq\frac 1{e^{-2nt}\lambda(\{G_{\Omega,w}<t\})}.
\end{equation} 
Here 
  $$K_\Omega(w)=\sup\{|f(w)|^2:f\in\mathcal O(\Omega),\ 
       \int_\Omega|f|^2d\lambda\leq 1\}$$
and
  $$G_{\Omega,w}=\sup\{u\in PSH^-(\Omega):u\leq\log|\cdot-w|+C
      \text{ near }w\}.$$
The constant in \eqref{est1} is optimal for every $t$, for
example we have the equality if $\Omega$ is a ball centered at $w$.
The behaviour of the right-hand side of \eqref{est1} as 
$t\to-\infty$ seems of particular interest. For example for 
$n=1$ we easily have
\begin{equation}\label{lim}
\lim_{t\to-\infty}e^{-2t}\lambda(\{G_{\Omega,w}<t\})
     =\frac\pi{(c_\Omega(w))^2},
\end{equation}
where 
  $$c_\Omega(w)=\exp\lim_{z\to w}\big(G_{\Omega,w}(z)
     -\log|z-w|\big)$$
is the logarithmic capacity of the complement of $\Omega$ 
with respect to $w$. This gave another proof in \cite{B5} 
of the Suita conjecture \cite{S} 
\begin{equation}\label{suita}
c_\Omega^2\leq\pi K_\Omega,
\end{equation}
originally shown in \cite{B4}.

Our first result is a counterpart of \eqref{lim} in higher 
dimensions:

\begin{thm} \label{thm1}
Assume that $\Omega$ is a bounded hyperconvex
domain in $\mathbb C^n$. Then 
  $$\lim_{t\to-\infty}e^{-2nt}\lambda(\{G_{\Omega,w}<t\})
     =\lambda(I^A_\Omega(w)),$$
where
  $$I^A_\Omega(w)=\{X\in\mathbb C^n:
     \varlimsup_{\zeta\to 0}\big(G_{\Omega,w}(w+\zeta X)
       -\log|\zeta|\big)<0\}$$
is the Azukawa indicatrix of $\Omega$ at $w$.
\end{thm}

It would be interesting to generalize this to a bigger class
of domains. Combining \eqref{est1} with Theorem \ref{thm1}
and approximating pseudoconvex domains by hyperconvex ones
from inside we obtain the following 
multidimensional version of the Suita conjecture:

\begin{thm}\label{thm2} 
For a pseudoconvex domain $\Omega$ in 
$\mathbb C^n$ and $w\in\Omega$ we have
\begin{equation} \label{msuita}
K_\Omega(w)\geq\frac 1{\lambda(I^A_\Omega(w))}.
\end{equation}
\end{thm}

Possible monotonicity of convergence in Theorem \ref{thm1}
is an interesting problem. We state the following:

\begin{conj}\label{conj1}
If $\Omega$ is pseudoconvex in $\mathbb C^n$ then 
the function 
  $$t\longmapsto e^{-2nt}\lambda(\{G_{\Omega,w}<t\})$$
is non-decreasing on $(-\infty,0]$.
\end{conj}

We will show the following result:

\begin{thm} \label{tn1}
Conjecture \ref{conj1} is true for $n=1$.
\end{thm}

The main tool will be the isoperimetric inequality. In fact,
the proof of Theorem \ref{tn1} will show that Conjecture 
\ref{conj1} in arbitrary dimension is equivalent to the following 
{\it pluricomplex isoperimetric inequality}: 
  $$\int_{\partial\Omega}\frac{d\sigma}{|\nabla G_{\Omega,w}|}
    \geq 4n\pi\lambda(\Omega)$$
for bounded strongly pseudoconvex $\Omega$ with smooth
boundary (by \cite{B1} the left-hand side is then well defined).

The following conjecture would easily give an affirmative
answer to Conjecture \ref{conj1}:

\begin{conj}\label{conj2}
If $\Omega$ is pseudoconvex in $\mathbb C^n$ then 
the function 
  $$t\longmapsto\log\lambda(\{G_{\Omega,w}<t\})$$
is convex on $(-\infty,0]$.
\end{conj}

Unfortunately, we do not know if it is true even for $n=1$.

In \cite{B2} the question was raised whether for $n=1$ a reverse
inequality to \eqref{suita}
  $$K_\Omega\leq Cc_\Omega^2$$
holds for some constant $C$. We answer it here in the negative:

\begin{prop}\label{ann} 
Assume that $0<r<1$ and let $P_r=\{z\in\mathbb C:r<|z|<1\}$.
Then
\begin{equation}\label{ann1}
\frac{K_\Omega(\sqrt r)}{(c_\Omega(\sqrt r))^2}
      \geq\frac{-2\log r}{\pi^3}.
\end{equation}
\end{prop}

It is nevertheless still plausible that there is an
upper bound for the Bergman kernel in terms of logarithmic
capacity which would give a quantitative version 
of the well known fact that for domains in $\mathbb C$
whose complement is a polar set the Bergman kernel vanishes.
The opposite implication is also well known and the 
quantitative version of this is given by \eqref{suita}.

There is however a class of domains for which the upper
bound does hold:

\begin{thm}\label{ub} For a $\mathbb C$-convex domain $\Omega$ 
in $\mathbb C^n$ and $w\in\Omega$ one has
  $$K_\Omega(w)\leq\frac{C^n}{\lambda(I^A_\Omega(w))}$$
with $C=16$. If $\Omega$ is convex then the estimate holds with
$C=4$ and if it is in addition symmetric with respect to $w$
then we can take $C=16/\pi^2$.
\end{thm}

By Theorems \ref{thm2} and \ref{ub} for $\mathbb C$-convex 
domains the function
  $$F_\Omega(w):=\big(K_\Omega(w)\lambda(I^A_\Omega(w))\big)^{1/n}$$
defined for $w\in\Omega$ with $K_\Omega(w)>0$, satisfies
\begin{equation}\label{i4}
1\leq F_\Omega\leq 16.
\end{equation}
One can easily check that $F_\Omega$ is biholomorphically invariant.  
If $\Omega$ is pseudoconvex and balanced with respect to $w$
(that is $w+z\in\Omega$ implies $w+\zeta z\in\Omega$ for 
$\zeta\in\bar\Delta$, where $\Delta$ is the unit disk)
then $F_\Omega(w)=1$. In fact a symmetrized bidisk
  $$\mathbb G_2=\{(\zeta_1+\zeta_2,\zeta_1\zeta_2):
     \zeta_1,\zeta_2\in\Delta\},$$
is an example of a $\mathbb C$-convex domain (see \cite
{NPZ1}) with $F_\Omega\not\equiv 1$. By \cite{EZ} we have 
$K_{\mathbb G_2}(0)=2/\pi^2$ and by \cite{AY} 
  $$I_{\mathbb G_2}^A(0)=\{X\in\mathbb C^2:
     |X_1|+2|X_2|<2\}.$$
Therefore $\lambda(I_{\mathbb G_2}^A(0))=2\pi^2/3$
and $F_{\mathbb G_2}(0)=2/\sqrt 3=1.15470\dots$

Especially interesting is the class of convex domains. It is well 
known that then the closure of the Azukawa indicatrix is equal to the 
Kobayashi indicatrix
  $$I^K_\Omega(w)=\{\varphi'(0):\varphi\in
      \mathcal O(\Delta,\Omega),\ \varphi(0)=w\}.$$
This follows from
Lempert's results \cite{L}, see \cite{JP}. For such domains 
the inequality $F_\Omega\geq 1$ was proved in \cite{B5} and 
seems very accurate. It is in fact much more difficult than for 
$\mathbb C$-convex domains to compute an example where one does 
not have equality. This can be done for some convex complex 
ellipsoids:

\begin{thm} \label{ell}
For $n\geq 2$ and $m\geq 1/2$ define
\begin{equation}\label{ell1}
\Omega=\{z\in\mathbb C^n:|z_1|+|z_2|^{2m}+\dots+|z_n|^{2m}<1\}.
\end{equation}
Then for $w=(b,0,\dots,0)$, where $0<b<1$, one has
\begin{equation}\label{form}
  K_\Omega(w)\lambda(I^K_\Omega(w))
     =1+(1-b)^a\frac{(1+b)^a-(1-b)^a-2ab}
     {2ab(1+b)^a},
\end{equation}
where $a=(n-1)/m+2$.
\end{thm}

For example, Theorem \ref{ell} gives the following graphs of 
$F_\Omega(b,0,\dots,0)$ for $m=1/2$ and $2\leq n\leq 6$
\footnote{Figures were done using {\it Mathematica}.}:

\vskip 10pt

\centerline{\includegraphics[scale=1.1]{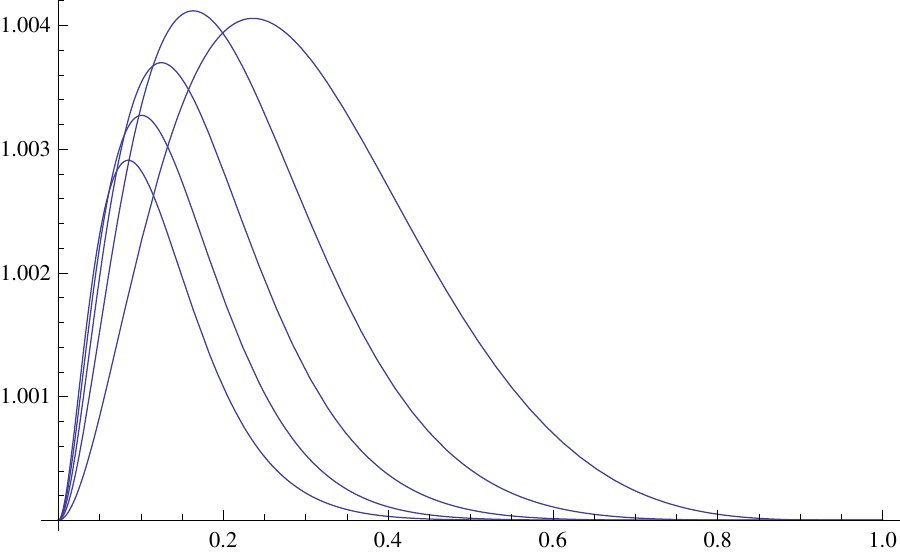}}

\noindent 
One can check numerically that the highest value 
of $F_\Omega(b,0,\dots,0)$ is attained for
$m=1/2$, $n=3$ at $b=0.163501\dots$, and is equal to 
$1.004178\dots$ 

Using \cite{6a} one can compute numerically 
$F_\Omega(b,0)$ for the ellipsoid
  $$\Omega=\{z\in\mathbb C^2:|z_1|^{2m}+|z_2|^2<1\},$$
where $m\geq 1/2$. This has an advantage compared to
the ellipsoid given by \eqref{ell1} because using 
holomorphic automorphisms we can easily show that 
all values of $F_\Omega$ are attained at $(b,0)$,
where $0<b<1$.
Here is the graph of $F_\Omega(b,0)$ for $m$ equal 
to 1/2, 2, 8, 32, and 128:

\vskip 10pt

\centerline{\includegraphics[scale=1.1]{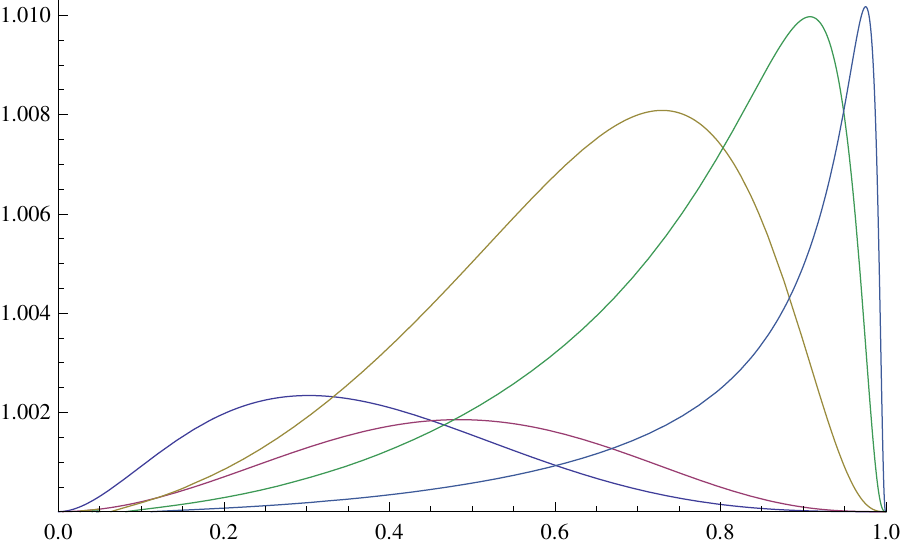}}

\noindent 
One can compute that the maximum converges to 
$1.010182\dots$ as $m\to\infty$. This is the highest
value of $F_\Omega$ for convex $\Omega$ we have been able
to obtain so far. It would be interesting to find
an optimal upper bound for $F_\Omega$ when $\Omega$ is
convex, how close to 1 it really is. We suspect that it 
is attained  for the ellipsoid
  $$\{z\in\mathbb C^n:|z_1|+\dots+|z_n|<1\}$$
at a point of the form $w=(b,\dots,b)$.

\begin{conj} Let $\Omega$ be convex and $w\in\Omega$ 
be such that $K_\Omega(w)>0$. Then $F_\Omega(w)=1$ if and only if 
there exists a balanced domain $\Omega'$ (not necessarily
convex) and a biholomorphic mapping 
$H:\Omega\rightarrow\Omega'$ such that $H(w)=0$.
\end{conj}

It was recently shown in \cite{GZ} that the equality 
holds in \eqref{suita} if and only if $\Omega$ 
is biholomorphic to $\Delta\setminus K$ for some
closed polar subset $K$, this was also conjectured 
by Suita in \cite{S}.

The paper is organized as follows: in Section \ref{s1} we 
show Theorems \ref{thm1} and \ref{tn1}. Upper bounds for the
Bergman kernel are discussed in Section \ref{s2}, we prove 
Proposition \ref{ann} and Theorem \ref{ub} there. Finally, in Section 
\ref{s3} the case of convex complex ellipsoids is treated. 

\section{Sublevel Sets of the Green Function \label{s1}}

\begin{proof}[Proof of Theorem \ref{thm1}] 
Without loss of generality we may assume that $w=0$. Write
$G:=G_{\Omega,0}$ and for $t\leq 0$ set 
  $$I_t:=e^{-t}\{G<t\}.$$
We can find $R>0$ such that $\Omega\subset B(0,R)$. Then
$\log(|z|/R)\leq G$ and $I_t\subset B(0,R)$. 
In our case by \cite{Z1} the function
  $$A(X)=\varlimsup_{\zeta\to 0}\big(G(\zeta X)
         -\log|\zeta|\big)$$
is continuous on $\mathbb C^n$ and $\varlimsup$ 
is equal to $\lim$. Therefore
  $$A(X)=\lim_{t\to-\infty}\big(G(e^tX)-t\big)$$
and by the Lebesgue bounded convergence theorem
  $$\lim_{t\to-\infty}\lambda(I_t)=\lambda(\{A<0\}).$$
\phantom{a}\vskip -37pt
\end{proof}

\begin{proof}[Proof of Theorem \ref{tn1}]
Set
  $$f(t):=\log\lambda(\{G<t\})-2t,$$
where  $G=G_{\Omega,w}$. It is enough to show that if
$t$ is a regular value of $G$ then $f'(t)\geq 0$. We have
  $$f'(t)=\frac{\displaystyle\frac d{dt}\lambda(\{G<t\})}
      {\lambda(\{G<t\})}-2.$$
The co-area formula gives
  $$\lambda(\{G<t\})=\int_{-\infty}^t\int_{\{G=s\}}
     \frac{d\sigma}{|\nabla G|}ds$$
and therefore
  $$\frac d{dt}\lambda(\{G<t\})=
     \int_{\{G=t\}}\frac{d\sigma}{|\nabla G|}.$$
By the Cauchy-Schwarz inequality
  $$\frac d{dt}\lambda(\{G<t\})\geq
    \frac{(\sigma(\{G=t\}))^2}{\displaystyle
      \int_{\{G=t\}}|\nabla G|d\sigma}
    =\frac{(\sigma(\{G=t\}))^2}{2\pi}.$$
The isoperimetric inequality gives
  $$(\sigma(\{G=t\}))^2\geq 4\pi\lambda(\{G<t\})$$
and we obtain $f'(t)\geq 0$.
\end{proof}

\section{Upper Bound for the Bergman kernel\label{s2}}

We first show that the reverse estimate to \eqref{msuita}
is not true in general.

\begin{proof}[Proof of Proposition \ref{ann}]
Since $z^j$, $j\in\mathbb Z$, is an orthogonal system
in $H^2(P_r)$ and
  $$||z^j||^2=\begin{cases}\dis\frac\pi{j+1}\big(1-r^{2j+2}\big)\
     &j\neq -1,\\ -2\pi\log r&j=-1,\end{cases}$$
we have
  $$K_{P_r}(w)=\frac 1{\pi|w|^2}\left(\frac 1{-2\log r}
      +\sum_{j\in\mathbb Z}\frac{j|w|^{2j}}{1-r^{2j}}\right)$$
and
\begin{equation}\label{kann}
K_{P_r}(\sqrt r)\geq\frac 1{-2\pi r\log r}.
\end{equation}
To estimate $c_{P_r}$ from above consider the mapping
  $$p(\zeta)=\exp\left(\frac{\log r}{\pi i}\Log
     \left(i\frac{1+\zeta}{1-\zeta}\right)\right),\ \ \ 
     \zeta\in\Delta,$$
where $\Log$ is the principal branch of the logarithm defined 
on $\mathbb C\setminus(-\infty,0]$. We have $p(0)=\sqrt r$
and $p'(0)=-2i\sqrt r\log r/\pi$. Also
  $$G_{P_r}(p(\zeta),\sqrt r)\leq\log|\zeta|$$
and therefore
  $$c_{P_r}(\sqrt r)\leq\frac 1{|p'(0)|}
      =\frac\pi{-2\sqrt r\log r}.$$
Combining this with \eqref{kann} we get \eqref{ann1}.
\end{proof}

Next, we show the reverse inequality to \eqref{msuita} for 
$\mathbb C$-convex domains.

\begin{proof}[Proof of Theorem \ref{ub}]
Write $I=I^A_\Omega(w)$. We may assume that 
$w=0$. We claim that it is enough to show that
\begin{equation}\label{cont}
I\subset\sqrt C\,\Omega.
\end{equation}
Indeed, since $I$ is balanced we would then have
  $$K_\Omega(0)\leq K_{I/\sqrt C}(0)
      =\frac 1{\lambda(I/\sqrt C)}=\frac{C^n}{\lambda(I)}.$$
The proof of \eqref{cont} will be similar to the proof of 
Proposition 1 in \cite{NPZ}. Choose $X\in I$ and by $L$ denote
the complex line generated by $X$. Let 
$a$ be a point from $L\cap\partial\Omega$ with the smallest
distance to the origin. We can find a hyperplane $H$ in 
$\mathbb C^n$ such that $H\cap\Omega=\emptyset$ (cf. \cite{H},
Theorem 4.6.8). Let $D$ be the set of those $\zeta\in\mathbb C$ 
such that $\zeta X$ belongs to the projection of $\Omega$ on $L$
along $H$. Then $D$ is a simply connected domain (cf.
\cite{H}, Proposition 4.6.7). Let $\varphi$ be a biholomorphic
mapping $\Delta\rightarrow D$ such that $\varphi(0)=0$.
We then have
  $$0>\varlimsup\big(G_{\Omega,0}(\zeta X)-\log|\zeta|\big)
     \geq\varlimsup\big(G_{D,0}(\zeta)-\log|\zeta|\big)
     =-\log|\varphi'(0)|.$$
By the Koebe quarter theorem $|\varphi'(0)|\leq 4r$, 
where $r$ is the distance from the origin to $\partial D$.
Since $r=|a|/|X|$, we obtain $|X|<4|a|$. This gives \eqref{cont} 
for $\mathbb C$-convex domains with $C=16$. If $\Omega$ is convex 
then so is $D$ and we may assume that it is a half-plane.
Then $|\varphi'(0)|\leq 2r$ and we get \eqref{cont} with $C=4$.
Finally, if $\Omega$ is symmetric then we may assume that $D$
is a strip centered at the origin and we get 
$|\varphi'(0)|\leq 4r/\pi$.
\end{proof}

\section{Complex Ellipsoids \label{s3}}

We first recall a general formula from \cite{JPZ} (it is
in fact a consequence of Lempert's theory \cite{L}) for geodesics 
in convex complex ellipsoids 
  $$\mathcal E(p)
     =\{z\in\mathbb C^n: |z_1|^{2p_1}+\dots+|z_n|^{2p_n}<1\},$$
where $p=(p_1,\dots,p_n)$, $p_j\geq 1/2$. For $A\subset\{1,\dots,n\}$ 
holomorphic mappings $\varphi:\Delta\rightarrow\mathcal E(p)$ of the 
form 
\begin{equation}\label{geod}
  \varphi_j(\zeta)
     =\begin{cases}\dis a_j\frac{\zeta-\alpha_j}{1-\bar\alpha_j\zeta}
       \left(\frac{1-\bar\alpha_j\zeta}
           {1-\bar\alpha_0\zeta}\right)^{1/p_j},\ &j\in A\\ \\
      \dis a_j\left(\frac{1-\bar\alpha_j\zeta}
         {1-\bar\alpha_0\zeta}\right)^{1/p_j},
              &j\notin A\end{cases},
\end{equation}
where $a_j\in\mathbb C_\ast$, $\alpha_j\in\Delta$ for $j\in A$, 
$\alpha_j\in\bar\Delta$ for $j\notin A$,
  $$\alpha_0=|a_1|^{2p_1}\alpha_1+\dots+|a_n|^{2p_n}\alpha_n,$$
and
  $$1+|\alpha_0|^2=|a_1|^{2p_1}(1+|\alpha_1|^2)+\dots
     +|a_n|^{2p_n}(1+|\alpha_n|^2),$$
form the set of almost all geodesics in $\Omega$ (possible
exceptions form a lower-dimensional set). A component $\varphi_j$
has a zero in $\Delta$ if and only if $j\in A$. We have
\begin{equation*}
  \varphi_j(0)
     =\begin{cases}\dis -a_j\alpha_j,\ &j\in A\\
      \dis a_j, &j\notin A\end{cases},
\end{equation*}
and
\begin{equation*}
  \varphi_j'(0)
     =\begin{cases}\dis a_j\left(1+\big(\frac 1{p_j}-1\big)|
        \alpha_j|^2-\frac{\alpha_j\bar\alpha_0}{p_j}\right),\ 
            &j\in A\\ \\
      \dis a_j\frac{\bar\alpha_0-\bar\alpha_j}{p_j},
              &j\notin A\end{cases}.
\end{equation*}
For $w\in\mathcal E(p)$ the set of vectors $\varphi'(0)$ where 
$\varphi(0)=w$ forms a subset of 
$\partial I^K_{\mathcal E(p)}(w)$ of a full measure.

Now assume that $w=(b,0,\dots,0)$. There are two possibilities: either 
$A=\{1,\dots,n\}$ or $A=\{2,\dots,n\}$. 
Since $\varphi(0)=w$, it follows that $\alpha_2=\dots=\alpha_n=0$, 
hence $\alpha_0=|a_1|^{2p_1}\alpha_1$ and
\begin{equation}\label{aal}
  1+|a_1|^{4p_1}|\alpha_1|^2=|a_1|^{2p_1}(1+|\alpha_1|^2)+|a_2|^{2p_2} 
      +\dots+|a_n|^{2p_n}.
\end{equation}
Moreover,
  $$\begin{cases} a_1\alpha_1=-b,\
       &1\in A\\ a_1=b,&1\notin A\end{cases}.$$
We will get vectors $X=\varphi'(0)$ from $\partial I^K_{\mathcal E(p)}(w)$, 
where
\begin{equation}\label{x1}
X_1=\begin{cases}\dis-\frac b{\alpha_1}\left(1+\big(\frac 1{p_1}
      -1\big)|\alpha_1|^2-\frac{b^{2p_1}|\alpha_1|^{2-2p_1}}{p_1}\right),\ 
             &1\in A\\ \\
   \dis -\bar\alpha_1\frac{b(1-b)}{p_1},&1\notin A\end{cases}
\end{equation}
and $X_j=a_j$, $j=2,\dots,n$. By \eqref{aal} the parameters are related by
  $$|a_2|^{2p_2}+\dots+|a_n|^{2p_n}
       =\begin{cases}\dis (1-b^{2p_1}|\alpha_1|^{-2p_1})
         (1-b^{2p_1}|\alpha_1|^{2-2p_1}),\ 
            &1\in A\\ 
     (1-b^{2p_1})(1-b^{2p_1}|\alpha_1|^2),&1\notin A\end{cases}.$$
If now $p_1=1/2$ as in Theorem \ref{ell} then by \eqref{x1}
  $$|\alpha_1|=\begin{cases}\dis\frac{2b^2+|X_1|-
      \sqrt{(2b^2+|X_1|)^2-4b^2}}{2b},\ &1\in A\\ \\
   \dis \frac{|X_1|}{2b(1-b)},&1\notin A\end{cases}.$$
After simple transformation we will obtain the following result:

\begin{thm}\label{ell8}
Assume that $p_1=1/2$, $p_j\geq 1/2$ for $j\geq 2$, and $0<b<1$. 
Then
  $$I_{\mathcal E(p)}^K((b,0,\dots,0))=\{X\in\mathbb C^n:
     |X_2|^{2p_2}+\dots+|X_n|^{2p_n}\leq\gamma(|X_1|)\},$$
where 
  $$\gamma(r)=\begin{cases}\dis 1-b-\frac{r^2}{4b(1-b)},\ 
     &\dis r\leq 2b(1-b)\\ \\ 
      \dis 1-b^2-r,\ 
           &\dis r>2b(1-b)\end{cases}.$$
\phantom{a}\vskip -31.5pt\qed
\end{thm}

\begin{proof}[Proof of Theorem \ref{ell}]
Denoting
  $$\omega=\lambda(\{z\in\mathbb C^{n-1}:|z_1|^{2m}+\dots+
      |z_{n-1}|^{2m}<1\}$$
we will get from Theorem \ref{ell8}
\begin{equation}\label{vol}
\begin{aligned}
\lambda(I_\Omega^K((b,0,\dots,0)))
  &=2\pi\omega\int_0^{1-b^2}r(\gamma(r))^{(n-1)/m}dr\\
  &=2\pi\omega(1-b)^a\frac{(1-b)^a+2ab}{a(a-1)}.
\end{aligned}
\end{equation}
It remains to compute the Bergman kernel. By the deflation
method from \cite{BFS} we obtain
  $$K_\Omega((b,0,\dots,0))=\frac{\lambda(\mathcal E(1/2,m/(n-1)))}
      {\lambda(\Omega)}K_{\mathcal E(1/2,m/(n-1))}((b,0)).$$
By Example 12.1.13 in \cite{JP} (see also formula (9) in
\cite{BFS}) 
  $$K_{\mathcal E(1/2,1/p)}((b,0))=\frac{p+1}{4\pi^2b}
      \big((1-b)^{-p-2}-(1+b)^{-p-2}\big).$$
We also have $\lambda(\mathcal E(1/2,1/p)=2\pi^2/((p+1)(p+2))$
and $\lambda(\Omega)=2\pi\omega/(a(a-1))$. It follows that
  $$K_\Omega((b,0,\dots,0))=\frac{a-1}{4\pi\omega b}
    \big((1-b)^{-a}-(1+b)^{-a}\big)$$
and combining this with \eqref{vol} gives \eqref{form}.
\end{proof}

\end{document}